\documentclass[a4paper]{amsart}
\usepackage[latin1]{inputenc}
\usepackage{amssymb}
\usepackage{amsmath}
\usepackage{mathrsfs}
\usepackage{amsthm}
\usepackage{amsfonts}
\usepackage{textcomp}
\usepackage{graphicx}
\usepackage[pdftex]{color}
\usepackage{paralist}
\usepackage[shortlabels]{enumitem}
\usepackage{hyperref}
\usepackage{comment}
\usepackage[arrow, matrix, curve]{xy}
\usepackage{tikz}

\newtheorem*{corollary*}{Corollary}
\newtheorem*{conjecture*}{Conjecture}
\newtheorem*{example*}{Example}
\newtheorem*{theorem*}{Theorem}
\newtheorem*{proposition*}{Proposition}

\newtheorem{theorem}{Theorem}[section]

\newtheorem{corollary}[theorem]{Corollary}
\newtheorem{lemma}[theorem]{Lemma}

\newtheorem*{claim*}{Claim}

\theoremstyle{definition}

\newtheorem{example}[theorem]{Example}

\theoremstyle{remark}

\numberwithin{equation}{section}
\makeatletter
\@namedef{subjclassname@2020}{%
  \textup{2020} Mathematics Subject Classification}
\makeatother
\makeatletter
\renewcommand*\env@matrix[1][\
arraystretch]{%
  \edef\arraystretch{#1}%
  \hskip -\arraycolsep
  \let\@ifnextchar\new@ifnextchar
  \array{*\c@MaxMatrixCols c}}
\makeatother

\renewcommand{\mod}{\operatorname{mod}}
\newcommand{\proj}{\operatorname{proj}}

\newcommand{\Ext}{\operatorname{Ext}}

\newcommand{\gldim}{\operatorname{gldim}}
\newcommand{\End}{\operatorname{End}}

\newcommand{\pdim}{\operatorname{pdim}}

\newcommand{\gr}{\operatorname{grade}}

\newcommand{\Hom}{\operatorname{Hom}}

\newcommand{\add}{\operatorname{\mathrm{add}}}
\renewcommand{\top}{\operatorname{\mathrm{top}}}

\newcommand{\Gr}{\operatorname{Gr}}
\begin{document}

\title{On the interaction of the Coxeter transformation and the rowmotion bijection}
\date{\today}



\author{Ren\'{e} Marczinzik}
\address{Mathematical Institute of the university of Bonn, University of Bonn, Endenicher Allee 60, 53115 Bonn, Germany}
\email{marczire@math.uni-bonn.de}
\author{Hugh Thomas}
\address{LACIM, PK-4211, Universit\'e du Qu\'ebec \`a Montr\'eal
CP 8888, Succ. Centre-ville
Montr\'eal (Qu\'ebec) H3C 3P8 Canada}
\email{thomas.hugh\_r@uqam.ca}
\author{Emine Y\i{}ld\i{}r\i{}m}
\address{The University of Cambridge, Department of Pure Mathematics and Mathematical Statistics, Centre for Mathematical Sciences, Wilberforce Road, Cambridge, CB3 0WB, United Kingdom}
\email{ey260@cam.ac.uk}

\begin{abstract}
Let $P$ be a finite poset and $L$ the associated distributive lattice of order ideals of $P$.
Let $\rho$ denote the rowmotion bijection of the order ideals of $P$ viewed as a permutation matrix and $C$ the Coxeter matrix for the incidence algebra $kL$ of $L$.
Then we show the identity $(\rho^{-1} C)^2=id$, as was originally conjectured by Sam Hopkins. Recently it was noted that the rowmotion bijection is a special case of the much more general grade bijection $R$ that exists for any Auslander regular algebra. This motivates to study the interaction of the grade bijection and the Coxeter matrix for general Auslander regular algebras. For the class of higher Auslander algebras coming from $n$-representation finite algebras we show that $(R^{-1} C)^2=id$ if $n$ is even and $(R^{-1}C+id)^2=0$ when $n$ is odd.
\end{abstract}

\maketitle

\section{Introduction}\label{sone}
Let $A$ be a finite dimensional algebra over a field $k$ with finite global dimension. We will always assume that $k$ is a splitting field for the algebra $A$, which for example is true if $k$ algebraically closed or if $A$ is a quiver algebra.
We denote the indecomposable projective $A$-modules by $P_i$ for $i=1,...,n$. Then the \emph{Cartan matrix} $M$ of $A$ is defined as the $n \times n$-matrix with entries $m_{i,j}:=\dim_k \Hom_A(P_i, P_j)$. The \emph{Coxeter matrix} (a.k.a. Coxeter transformation) $C$ of $A$ is then defined as $C:=-M^{-1} M^T$. Note that this is well defined as the Cartan matrix of an algebra of finite global dimension has determinant 1 or -1 and thus $M$ is invertible over $\mathbb{Z}$, see for example \cite[Proposition III.3.10]{ASS}

The Coxeter matrix of a finite dimensional algebra with finite global dimension is the main object of study in the spectral theory of finite dimensional algebras, we refer for example to the survey article \cite{LP}. Of special interest in homological algebra are algebras with periodic Coxeter matrix, that is $C^l=id$ for some $l \geq 1$. Algebras with periodic Coxeter matrix include for example fractionally Calabi-Yau algebras that arise in geometric and combinatorial contexts, we refer for example to \cite{CDIM} and \cite{Y}.

When $P$ is a finite poset then the set of order ideals of $P$ is a distributive lattice $L$ and every finite distributive lattice arises this way and is uniquely determined by $P$.
The Coxeter matrix of the incidence algebra $kL$ can be directly read off from elementary combinatorial data involving the M\"obius function of $L$.
When $S$ is an antichain of $P$ then set $I(S)$ to be the order ideal whose maximal elements are given by $S$ and set $M(S)$ to be the order ideal whose minimal non-elements are given by $S$. The \emph{rowmotion bijection} is given on the elements of $L$ by sending $I(S)$ to $M(S)$, which defines a bijection. The rowmotion bijection is one of the main attractions in dynamical algebraic combinatorics, we refer for example to the articles \cite{H}, \cite{Str} and \cite{TW}. When $L$ has $n$ elements, we can associate to the rowmotion bijection an $n \times n$-permutation matrix in a natural way.
While the Coxeter matrix and the rowmotion bijection are well studied objects in representation theory and combinatorics respectively, it seems that no relation between them has been shown before.

The following theorem, which is our first main result, was first conjectured by Sam Hopkins, who noted the identity in several examples:
\begin{theorem}
Let $L$ be a finite distributive lattice with $C$ the Coxeter matrix of the incidence algebra of $kL$ and $\rho$ the rowmotion bijection for $L$ viewed as a permutation matrix.
Then $\rho^{-1} C$ has order two, that is $(\rho^{-1} C)^2=id$.

\end{theorem}

In the article \cite{IM} it was noted that the rowmotion bijection for a distributive lattice $L$ is a special case of a much more general bijection that exists for any Auslander regular algebra.
Recall here that a finite dimensional algebra $A$ is called \emph{Auslander regular} if $A$ has finite global dimension and in the minimal injective coresolution 
$$0 \rightarrow A \rightarrow I_0 \rightarrow I_1 \rightarrow \cdots I_n \rightarrow 0$$
we have that the projective dimension of $I_i$ is bounded by $i$ for all $i \geq 0$.
In \cite{I}, Iyama defined the so-called \emph{grade bijection} of an Auslander
regular algebra $A$. It is a permutation of the simple modules of $A$, whose precise definition we will recall in the next section. 
Again, we can associate a permutation matrix to the grade bijection that we will usually denote by $R_A$, or $R$ for short, for an Auslander regular algebra $A$.
One of the main results in \cite{IM} states that a finite lattice $L$ is distributive if and only if the incidence algebra $kL$ is Auslander regular. 
Moreover, when $L$ is distributive the grade bijection gives a homological realisation of the rowmotion bijection, when one identifies the elements of the lattice $L$ with the simple $kL$-modules in a natural way.
This leads to the natural question, whether for other Auslander regular algebras
there is a simple relation between the Coxeter matrix and the grade bijection.

We give a positive answer for an important class of algebras that also appears in combinatorics. Namely, an algebra $A$ is called $n$-representation-finite if $A$ has global dimension at most $n$ and there is an $n$-cluster tilting object $M$ in $\mod A$, which is then uniquely given when we assume that $M$ is basic. 
The notion of $n$-representation finite algebras was introduced in \cite{Iya2} and includes many important classes of algebras such as all path algebras of Dynkin type, higher Auslander algebras of Dynkin type $A$ that have strong relations to the combinatorics of cyclic  polytopes \cite{OT} and the 2-representation finite algebras that have strong relations to quivers with potential and Jacobian algebras \cite{HI}.  
For an $n$-representation finite $A$ with $n$-cluster tilting module $M$, the endomorphism algebra $B:=\End_A(M)$ will be an higher Auslander algebra. Higher Auslander algebras are those algebras $B$ with finite global dimension and a minimal injective coresolution 
$$0 \rightarrow B \rightarrow I_0 \rightarrow I_1 \rightarrow \cdots I_n \rightarrow 0$$ 
such that the projective dimension of $I_i$ is zero for $i=0,1,...,n-1$, when $n$ denotes the global dimension. Thus by definition higher Auslander algebras are always Auslander regular. Our second main result is as follows for such algebras $B$.
\begin{theorem}
Let $A$ be an $n$-representation finite algebra with $n$-cluster tilting module $M$.
Let $B=\End_A(M)$ with grade bijection $R$ and Coxeter matrix $C$.
If $n$ is even then $(R^{-1}C)^2=id$ and if $n$ is odd then $(R^{-1}C+id)^2=0$.
\end{theorem}

\section{Preliminaries}
We assume that algebras are finite dimensional over a field $k$ and that they are non-semisimple and connected unless stated otherwise. Additionally, we assume that $k$ is a splitting field, which is for example automatic if $k$ is algebraically closed or if $A$ is a quiver algebra. Here $k$ being a splitting field for the $k$-algebra $A$ means that every simple module $A$-module $S$ has the property that $\End_A(S) \cong k$, see for example \cite[Chapter 7]{Lam} for more equivalent characterisations and properties of splitting fields. We assume that modules are right modules unless otherwise stated. $D=\Hom_k(-,k)$ denotes the duality and $J$ denotes the Jacobson radical of an algebra. We assume that the reader is familiar with the basics of representation theory and homological algebra of finite dimensional algebras and we refer for example to the textbooks \cite{ARS} and \cite{SkoYam}.
Let $\nu_A:=D \Hom_A(-,A)$ denote the Nakayama functor of an algebra $A$ and $\nu_A^{-1}=\Hom_A(D(A),-)$ its inverse. It is well known that $\nu_A$ induces an equivalence between the category of projective $A$-modules and the category of injective $A$-modules with inverse $\nu_A^{-1}$.
The \emph{global dimension} of an algebra $A$ is defined as the supremum of all projective dimensions of $A$-modules. The \emph{dominant dimension} of an algebra $A$ with minimal injective coresolution 
$$0 \rightarrow A \rightarrow I_0 \rightarrow I_1 \rightarrow \cdots $$
is defined as the smallest $n \geq 0$ such that $I_n$ is not projective or as infinite if there is no such $I_n$. For an $A$-module $M$, $\add(M)$ will denote the full subcategory of $\mod A$ consisting of all direct summands of $M^n$ for some natural number $n$.
Let $K_0(A)$ denote the Grothendieck group of an algebra $A$ with finite global dimension with basis given by the indecomposable projective modules $[P]$. For simplicity, we will often omit the usual brackets $[P]$ for an element in the Grothendieck group $K_0(A)$ when there is no danger of confusion. When $A$ has finite global dimension, the \emph{Coxeter transformation} of $A$ is defined as $C_A([P]):=-[\nu_A(P)]$. When this linear transformation is expressed as a matrix with respect to the basis of $K_0(A)$ given by the classes of the indecomposable projective modules, we recover the matrix $C$ from Section \ref{sone}.

An algebra $A$ is called \emph{Auslander regular} when $A$ has finite global dimension and there exists an injective coresolution
$$0 \rightarrow A \rightarrow I^0 \rightarrow I^1 \rightarrow \cdots \rightarrow I^n \rightarrow 0$$
of the regular module such that $\pdim I^i \leq i$ for all $i \geq 0$.
Important classes of Auslander regular algebras are incidence algebras of distributive lattices (see \cite{IM}), higher Auslander algebras (see \cite{Iya4}) and blocks of category $\mathcal{O}$ (see \cite{KMM}).
The \emph{grade} of an $A$-module $M$ is defined by $\gr M:= \inf \{ i \geq 0 \mid \Ext_A^i(M,A) \neq 0 \}$. Dually, the \emph{cograde} of a module $M$ is defined as $\inf \{i \geq 0 \mid \Ext_A^i(D(A),M) \neq 0 \}$. For every Auslander regular algebra, there exists a bijection on the simple $A$-modules as follows:
\begin{theorem} (\cite[Theorem 2.10]{I}) \label{graderowmotiontheorem}
Let $A$ be an Auslander regular algebra.
Then there is a bijection $\Gr_A$ sending a simple $A$-module $S$ to the simple $A$-module $\Gr_A(S)=\top(D \Ext_A^{g_S}(S,A))$, where $g_S:= \gr S$.
The grade of $S$ equals the cograde of $\Gr_A(S)$.
\end{theorem}
We call the bijection $\Gr_A$ as in the previous theorem for Auslander regular algebras the \emph{grade bijection} on simple $A$-modules.
We refer to \cite[Theorem 2.10]{I} for a more general statement and proofs. 
In \cite{IM} it was shown that the grade bijection gives a categorification of the rowmotion map for incidence algebras of distributive lattices.
We define the grade bijection on the Grothendieck group of an Auslander regular algebra $A$ by $R_A([P_S])=[P_{\Gr_A(S)}]$, where $P_S$ denotes the projective cover of a simple module $S$. 
We define the \emph{rowmotion Coxeter transformation} of an Auslander regular algebra $A$ as $R_A^{-1} C_A$.  Note that since $R_A$ is a permutation matrix, we have $R_A^{-1}=R_A^T$. We will be mainly interested in the minimal polynomial of $R_A^{-1} C_A$ and the relation $(R_A^{-1} C_A)^2=id$. Note that since $  C_A R_A^{-1}= R_A (R_A^{-1} C_A) R_A^{-1}$, the operators $C_A R_A^{-1}$ and $R_A^{-1} C_A$ are similar and thus they have the same minimal polynomials and we have $(C_A R_A^{-1})^2=id$ if and only if $(R_A^{-1} C_A)^2=id$.

\section{The Coxeter transformation and the rowmotion bijection for distributive lattices}
Let $P$ denote a finite poset and $L$ the distributive lattice of order ideals of $P$.
Let $S$ be an antichain of $P$ and $I(S)$ the order ideal whose maximal elements are given by $S$ and $M(S)$ the order ideal whose minimal non-elements are given by $S$. Formally,

    \[I(S)=\{x\in P : x\leq s \text{ for some } s\in S\}. \]
    \[M(S)=P\setminus \{y\in P :y\geq s \text{ for some } s\in S\}.\]

The \emph{Hasse quiver} of a poset $P$ is the finite quiver with points $x \in P$ and an arrow $x \rightarrow y$ if $x$ covers $y$ (note that we use here the opposite convention compared to \cite{IM}). The \emph{incidence algebra} $kP$ of a poset $P$ over a field $k$ is defined as the quiver algebra $kQ/\mathcal I$ with $Q$ the Hasse quiver of $P$ and the relations $\mathcal I$ such that any two paths that start and end at the same points get identified.

The \emph{rowmotion bijection} $\rho$ for a distributive lattice $L$ given as the set of order ideals of a poset $P$ is defined as the permutation sending $I(S)$ to $M(S)$.
As explained in the preliminaries, the rowmotion bijection is a special case of the more general grade bijection when viewing the incidence algebra of a distributive lattice as an Auslander regular algebra. We will denote by $P_{O}$ the indecomposable projective module in the incidence algebra $kL$ of a distributive lattice $L$ corresponding to the order ideal $O$ and $J_O:=\nu_A(P_O)$ (we use $J$ instead of $I$ for the indecomposable injectives in this section to avoid confusion with $I(S)$). When $A=kL$ is the incidence algebra of a distributive lattice, we set $R_A([P_{I(S)}])=[P_{M(S)}]$ motivated by the fact that the grade bijection gives a homological realisation of the rowmotion bijection as explained after Theorem \ref{graderowmotiontheorem}.
We will use the following result, whose proof can be found in \cite[Theorem 3.2]{IM} and where one can also find a description of the differentials. See also \cite[Definition 2.3]{Y}, where this construction was anticipated, though restricted to a particular case.
\begin{theorem}\label{resn}
Let $L$ be a distributive lattice given as the set of order ideals of a poset $P$.
Then a minimal projective resolution of the indecomposable injective module $J_{I(S)}$ for an antichain $S$ is given as follows in $kL$:
$$0 \rightarrow P_{M(S)} \rightarrow \bigoplus_{T \subseteq S ,|T|=r-1}^{}{P_{M(T)}} \rightarrow \cdots \rightarrow \bigoplus_{T \subseteq S ,|T|=i}^{}{P_{M(T)}} \rightarrow \cdots \rightarrow \bigoplus_{T \subseteq S ,|T|=1}^{}{P_{M(T)}} \rightarrow P_{M(\emptyset)} \rightarrow J_{I(S)} \rightarrow 0.$$

\end{theorem}

\begin{example} For an example of the previous theorem, consider the case where $P$ is the poset shown on the left in the diagram below. The lattice $L$ is
  the distributive lattice of order ideals of $P$, shown on the right.

  $$
  \vcenter{\hbox{\begin{tikzpicture}[-stealth,xscale=1.3]
\node(a) at (0,0) {$x_4$};
\node (b) at (-1,-1) {$x_1$};
\node (c) at (0,-1) {$x_2$};
\node (d) at (1,-1) {$x_3$};
\draw (a) -- (b); \draw (a)--(c); \draw (a)--(d);
  \end{tikzpicture}}}
  \qquad\qquad
\vcenter{\hbox{  
\begin{tikzpicture}[stealth-,xscale=1.3] \node(a) at (0,0) {$\emptyset$};
    \node (b) at (-1,1) {$\{x_1\}$};
    \node (c) at (0,1) {$\{x_2\}$};
    \node (d) at (1,1) {$\{x_3\}$};
    \node (e) at (-1,2) {$\{x_1,x_2\}$};
    \node (f) at (0,2) {$\{x_1,x_3\}$};
    \node (g) at (1,2) {$\{x_2,x_3\}$};
    \node (h) at (0,3) {$\{x_1,x_2,x_3\}$};
    \node (j) at (0,4) {$\{x_1,x_2,x_3,x_4\}$};
    \draw (a) -- (b);\draw (a)--(c);\draw (a)--(d);\draw (b)--(e);\draw (b)--(f);\draw (c)--(e);\draw (c)--(g);\draw
    (d)--(f);\draw (d)--(g);\draw (e)--(h);\draw (f)--(h);\draw (g) --(h);
    \draw (h)--(j);
    \end{tikzpicture}
    }}
    $$

  Let $S=\{x_1,x_2\}$. The projective resolution of $J_{I(S)}$ given by Theorem
  \ref{resn} is as follows, where we represent $kL$ modules by the corresponding representations of the Hasse quiver.

  $$ 0 \rightarrow \vcenter{\hbox{\begin{tikzpicture}[stealth-,scale=.7] \node(a) at (0,0) {$k$};
    \node (b) at (-1,1) {$0$};
    \node (c) at (0,1) {$0$};
    \node (d) at (1,1) {$k$};
    \node (e) at (-1,2) {$0$};
    \node (f) at (0,2) {$0$};
    \node (g) at (1,2) {$0$};
    \node (h) at (0,3) {$0$};
    \node (k) at (0,4) {$0$};
    \draw (a) -- (b);\draw (a)--(c);\draw (a)--(d);\draw (b)--(e);\draw (b)--(f);\draw (c)--(e);\draw (c)--(g);\draw
    (d)--(f);\draw (d)--(g);\draw (e)--(h);\draw (f)--(h);\draw (g) --(h);
    \draw (h) --(k);
    \end{tikzpicture}}}
 \rightarrow 
\vcenter{\hbox{  \begin{tikzpicture}[stealth-,scale=.7] \node(a) at (0,0) {$k$};
    \node (b) at (-1,1) {$k$};
    \node (c) at (0,1) {$0$};
    \node (d) at (1,1) {$k$};
    \node (e) at (-1,2) {$0$};
    \node (f) at (0,2) {$k$};
    \node (g) at (1,2) {$0$};
    \node (h) at (0,3) {$0$};
    \node (k) at (0,4) {$0$};
    \draw (a) -- (b);\draw (a)--(c);\draw (a)--(d);\draw (b)--(e);\draw (b)--(f);\draw (c)--(e);\draw (c)--(g);\draw
    (d)--(f);\draw (d)--(g);\draw (e)--(h);\draw (f)--(h);\draw (g) --(h);
    \draw (h) --(k);
  \end{tikzpicture}}}
  \oplus
\vcenter{\hbox{  \begin{tikzpicture}[stealth-,scale=.7] \node(a) at (0,0) {$k$};
    \node (b) at (-1,1) {$0$};
    \node (c) at (0,1) {$k$};
    \node (d) at (1,1) {$k$};
    \node (e) at (-1,2) {$0$};
    \node (f) at (0,2) {$0$};
    \node (g) at (1,2) {$k$};
    \node (h) at (0,3) {$0$};
    \node (k) at (0,4) {$0$};
    \draw (a) -- (b);\draw (a)--(c);\draw (a)--(d);\draw (b)--(e);\draw (b)--(f);\draw (c)--(e);\draw (c)--(g);\draw
    (d)--(f);\draw (d)--(g);\draw (e)--(h);\draw (f)--(h);\draw (g) --(h);
    \draw (h) --(k);
  \end{tikzpicture}}}
  \rightarrow
\vcenter{\hbox{   \begin{tikzpicture}[stealth-,scale=.7] \node(a) at (0,0) {$k$};
    \node (b) at (-1,1) {$k$};
    \node (c) at (0,1) {$k$};
    \node (d) at (1,1) {$k$};
    \node (e) at (-1,2) {$k$};
    \node (f) at (0,2) {$k$};
    \node (g) at (1,2) {$k$};
    \node (h) at (0,3) {$k$};
    \node (k) at (0,4) {$k$};
    \draw (a) -- (b);\draw (a)--(c);\draw (a)--(d);\draw (b)--(e);\draw (b)--(f);\draw (c)--(e);\draw (c)--(g);\draw
    (d)--(f);\draw (d)--(g);\draw (e)--(h);\draw (f)--(h);\draw (g) --(h);\draw(h)--(k);
   \end{tikzpicture}}}
   \rightarrow
 \vcenter{\hbox{   \begin{tikzpicture}[latex-,scale=.7] \node(a) at (0,0) {$0$};
    \node (b) at (-1,1) {$0$};
    \node (c) at (0,1) {$0$};
    \node (d) at (1,1) {$0$};
    \node (e) at (-1,2) {$k$};
    \node (f) at (0,2) {$0$};
    \node (g) at (1,2) {$0$};
    \node (h) at (0,3) {$k$};
    \node (k) at (0,4) {$k$};
    \draw (a) -- (b);\draw (a)--(c);\draw (a)--(d);\draw (b)--(e);\draw (b)--(f);\draw (c)--(e);\draw (c)--(g);\draw
    (d)--(f);\draw (d)--(g);\draw (e)--(h);\draw (f)--(h);\draw (g) --(h);\draw (h)--(k);
    \end{tikzpicture}}}
    \rightarrow 0$$
\end{example}
  
\begin{corollary}
Let $L$ be a distributive lattice given as the set of order ideals of a poset $P$.
Let $S$ denote an antichain of $P$.
Then the Coxeter transformation $C_A$ for $A=kL$ is given by 
$$C_A( [P_{I(S)}])=-\sum\limits_{T \subseteq S}^{}{ (-1)^{|T|} [ P_{M(T)}}].$$
\end{corollary}

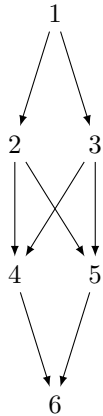
\begin{theorem}
Let $A=kL$ be the incidence algebra of a distributive lattice $L$ given as the set of order ideals of a poset $P$.
Then $(R_A^{-1} C_A)^2=id$.

\end{theorem}
\begin{proof}
We first look at $C_A R_A^{-1} C_A$ and determine the values on a basis:
\begin{eqnarray*} C_A R_A^{-1} C_A ([P_{I(S)}])&=& C_A R_A^{-1} (-\sum\limits_{T \subseteq S}^{}{ (-1)^{|T|} [ P_{M(T)}}])\\
&=& C_A ( -\sum\limits_{T \subseteq S}^{}{ (-1)^{|T|} [ P_{I(T)}}])\\
&=& \sum\limits_{T \subseteq S}^{}{ \sum\limits_{R \subseteq T}^{}{ (-1)^{|T|} (-1)^{|R|} [P_{M(R)}]}}\\
&=&\sum_{R\subseteq S} \sum_{R\subseteq T\subseteq S}(-1)^{|T|-|R|}[P_{M(R)}] \\
&=&\sum_{R\subseteq S} [P_{M(R)}]\sum_{R\subseteq T\subseteq S}(-1)^{|T|-|R|} 
\end{eqnarray*}
If $R$ is properly contained in $S$, the inner sum is over the $2^{|S|-|R|}$ choices for $T$, half with $|T|-|R|$ even, and half with $|T|-|R|$ odd. The inner sum is therefore zero in this case.
The only surviving term of the outer sum is $R=S$, in which case we also have $T=S$.
The result is that $C_A R_A^{-1} C_A([P_{I(S)}])=[P_{M(S)}]$ and if we now apply $R_A^{-1}$ to both sides, we get that $R_A^{-1} C_A R_A^{-1} C_A ([P_{I(S)}])=[P_{I(S)}]$, as desired.
\end{proof}

The next example shows that the identity $(R_A^{-1}C_A)^2=id$, which is true when $A$ is the incidence algebra of a distributive lattices, does not hold for Auslander regular incidence algebras of posets that are not lattices:
\begin{example} \label{examplenonlattice}
The following poset with incidence algebra $A$ is Auslander regular:
\begin{center}
\begin{tikzpicture}[>=latex,line join=bevel,]
\node (node_0) at (21.0bp,6.5bp) [draw,draw=none] {$6$};
  \node (node_1) at (6.0bp,55.5bp) [draw,draw=none] {$4$};
  \node (node_2) at (36.0bp,55.5bp) [draw,draw=none] {$5$};
  \node (node_3) at (6.0bp,104.5bp) [draw,draw=none] {$2$};
  \node (node_4) at (36.0bp,104.5bp) [draw,draw=none] {$3$};
  \node (node_5) at (21.0bp,153.5bp) [draw,draw=none] {$1$};
  \draw [black,<-] (node_0) ..controls (17.092bp,19.746bp) and (13.649bp,30.534bp)  .. (node_1);
  \draw [black,<-] (node_0) ..controls (24.908bp,19.746bp) and (28.351bp,30.534bp)  .. (node_2);
  \draw [black,<-] (node_1) ..controls (6.0bp,68.603bp) and (6.0bp,79.062bp)  .. (node_3);
  \draw [black,<-] (node_1) ..controls (13.953bp,68.96bp) and (21.16bp,80.25bp)  .. (node_4);
  \draw [black,<-] (node_2) ..controls (28.047bp,68.96bp) and (20.84bp,80.25bp)  .. (node_3);
  \draw [black,<-] (node_2) ..controls (36.0bp,68.603bp) and (36.0bp,79.062bp)  .. (node_4);
  \draw [black,<-] (node_3) ..controls (9.9083bp,117.75bp) and (13.351bp,128.53bp)  .. (node_5);
  \draw [black,<-] (node_4) ..controls (32.092bp,117.75bp) and (28.649bp,128.53bp)  .. (node_5);
\end{tikzpicture}
\end{center}
Let $P_i$ denote the indecomposable projective $A$-modules and $I_i$ the indecomposable injective $A$-modules.
  A computer program such as the GAP-package \cite{QPA} can be used to show that $A$ is Auslander regular and to find the grade bijection and Coxeter matrix.  We just give a sketch for the calculations without all details in the following. One finds the injective resolution of each indecomposable projective, and then checks the projective dimension of the injectives that appear. The minimal injective resolutions of indecomposable projectives are given below.

\[ 0\rightarrow P_1 \rightarrow I_6  \rightarrow 0, \]

\[ 0\rightarrow P_2 \rightarrow I_6  \rightarrow I_3 \rightarrow 0, \]

\[ 0\rightarrow P_3 \rightarrow I_6  \rightarrow I_2 \rightarrow 0, \]

\[ 0\rightarrow P_4 \rightarrow I_6  \rightarrow I_5 \rightarrow 0, \]

\[ 0\rightarrow P_5 \rightarrow I_6  \rightarrow I_4 \rightarrow 0, \]
      and
      
\[0\rightarrow P_6 \rightarrow I_6 \rightarrow I_4 \oplus I_5 \rightarrow I_2 \oplus I_3 \rightarrow I_1 \rightarrow 0.\]

Since the algebra is isomorphic to its opposite algebra, one can dually obtain the minimal projective resolution of the indecomposable injectives and from that the projective dimensions of the indecomposable injectives to see that $A$ is indeed Auslander regular.

We emphasize that $R_A$ here is given by the grade bijection. Currently, no purely combinatorial description of this bijection is available for Auslander regular incidence algebras of posets except in the case that the poset is a distributive lattice. Such a combinatorial description would be a generalization of rowmotion on distributive lattices, and would be very interesting to have.
In forthcoming work \cite{KMT} we show that the grade bijection $R_A$ of a general Auslander regular algebra $A$ is given by sending $P_i$ to the last term in the minimal projective resolution of $\nu_A(P_i)=I_i$, which allows to calculate the grade bijection directly from the minimal projective resolutions of the $I_i$ as above.

The matrix for $R_A$ is given as follows:
$$\left(\begin{array}{cccccc}
       0 & 0& 0&0&0&1\\
       0 & 0& 1&0&0&0\\
       0 & 1& 0&0&0&0\\
       0 & 0& 0&0&1&0\\
       0 & 0& 0&1&0&0\\
       1 & 0& 0&0&0&0\\
\end{array}\right)$$
\noindent
and the Coxeter matrix $C_A$ is given as follows:
$$\left(\begin{array}{rrrrrr}
-1 & -1& -1&-1&-1&-1\\
  1 & 0& 1&0&0&0\\
1 & 1& 0&0&0&0\\
-1 & 0& 0&0&1&0\\
-1 &0& 0&1&0&0\\
1 & 0& 0&0&0&0\\
\end{array}\right).$$

The matrix $R_A^{-1} C_A $ has minimal polynomial $x^3-x^2-x+1$ and thus the identity $(R_A^{-1} C_A)^2=id$ is not true for this poset.

\end{example}

\section{$n$-representation finite algebras}
We recall some basics on cluster tilting modules and higher Auslander algebras.
An $A$-module $M$ is called \emph{$n$-cluster-tilting} when $M$ is a generator-cogenerator and $M^{ \perp n}=\add(M)=^{\perp n}M$, where $M^{ \perp n}= \{ N \in \mod A | \Ext_A^i(M,N)=0 $ for all $1 \leq i < n \}$ and $^{\perp n}M= \{ N \in \mod A | \Ext_A^i(N,M)=0 $ for all $1 \leq i < n \}$. A 1-cluster tilting module $M$ exists if and only if $A$ is representation-finite; in this case $\add(M)=\mod A$.
By the higher Auslander correspondence, see for example \cite{Iya4}, $M$ is $n$-cluster tilting if and only if the algebra $B:=\End_A(M)$ is a \emph{higher Auslander algebra} of global dimension $n+1$, that is it has global dimension and dominant dimension equal to $n+1$. We denote by $\tau_n:=\tau \Omega^{n-1}$ for $n \geq 1$ the \emph{$n$-Auslander--Reiten translate} and by $\tau_n^{-1}:= \Omega^{-(n-1)}\tau^{-1}$ the \emph{inverse $n$-Auslander--Reiten translate}.
Let $A$ be an algebra and let $M$ be a generator-cogenerator of $\mod A$ and $B:=\End_A(M)$. In the following, we will summarize several properties of $B$ in this situation and refer for example to \cite[Chapter VI.5]{ARS} for more information. There is an equivalence of categories $\add(M) \cong \proj B$, given by $\Hom_A(M,-)$. We denote the indecomposable projective $B$-module associated to an indecomposable $A$-module $N \in \add(M)$ by $L_N:=\Hom_A(M,N)$. Note that $\nu_B(L_N)=D \Hom_B(L_N,B)=D \Hom_B(\Hom_A(M,N), \Hom_A(M,M)) \cong D \Hom_A(N,M)$ and thus the indecomposable injective $B$-modules are given by $T_N := D \Hom_A(N,M)$ for an indecomposable module $N \in \add(M)$. 
We denote by $S_N$ the simple $B$-module with projective cover $L_N$.
In this section the modules $P_i$ will refer to terms in a minimal projective resolution of a module and not to the indecomposable projective modules of an algebra corresponding.

An algebra $A$ is called \emph{$n$-representation-finite} for some $n \geq 1$ if $\gldim A \leq n$ and there is an $n$-cluster tilting module $M$ in $\mod A$, see for example \cite{Iya2} where such algebras were studied systematically for the first time.
Note that an $n$-cluster tilting module $M$ in a $n$-representation-finite algebra is unique when we assume that $M$ is basic. Note also that $M$
necessarily contains the indecomposable projective and injective modules as
direct summands. 
$1$-representation-finite algebras are exactly the representation finite hereditary algebras. They are classified by Dynkin diagrams, see 
\cite[Chapter VIII]{ARS}.

We collect several results that we will need. For a survey on $n$-cluster tilting categories and higher Auslander algebras, we refer to Section 2 of \cite{Iya5}. For the definition
of $n$-almost split sequences and their basic properties we refer to Section 2.3 in \cite{Iya5}, which also contains the next lemma. We recall that the length of an $n$-almost split sequence is equal to $n+2$.
\begin{lemma}
  Let $A$ be an algebra with $n$-cluster-tilting module $M$ and $B:=\End_A(M)$.
  \begin{enumerate} \item
Let $N$ be an indecomposable non-projective
summand of $M$. There is an $n$-almost split sequence
$$ 0 \rightarrow \tau_n(N) \rightarrow \dots \rightarrow N \rightarrow 0.$$
Applying $\Hom_A(M,-)$ to it yields a minimal projective resolution of $S_N$:
\begin{align}\label{SPres}
  0\rightarrow L_{\tau_n(N)} \rightarrow \cdots \rightarrow L_N
  \rightarrow S_N\rightarrow 0.\end{align}
\item Let $N$ be an indecomposable non-injective summand of $M$. There
  is an $n$-almost split sequence
  $$ 0 \rightarrow N \rightarrow \dots \rightarrow \tau_n^{-1}(N)
  \rightarrow 0.$$
  Applying $D\Hom_A(-,M)$, we obtain a minimal injective coresolution of
  $S_N$:
  \begin{align}\label{SIres}
  0\rightarrow S_N \rightarrow T_N
  \rightarrow \cdots \rightarrow T_{\tau_n^{-1}(N)}\rightarrow 0.\end{align}
\end{enumerate}
  \end{lemma}

  As an immediate consequence of this lemma, we deduce that if $N$ is
  an indecomposable summand of $M$ which is not projective, then the
  projective dimension of $S_N$ is $n+1$; similarly, if $N$ is not
  injective, then the injective dimension of $S_N$ is $n+1$.

\begin{lemma} \label{higherausreiformulas}
Let $A$ be an $n$-representation-finite algebra with $n$-cluster tilting module $M$. Let $X,Y \in \add(M)$.
\begin{enumerate}
\item $\Hom_A(\tau_n^{-1}(Y),X)=D \Ext_A^n(X,Y)$.
\item $\Hom_A(Y,\tau_n(X))=D \Ext_A^n(X,Y)$.
\end{enumerate}
\end{lemma}
\begin{proof}
We show (1), the proof of (2) is dual. 
We have in general that $\underline{\Hom}_A(\tau_n^{-1}(Y),X)=D \Ext_A^n(X,Y)$ for any $X,Y \in \add(M)$ for a cluster tilting module $M$ (without the assumption that $A$ has global dimension at most $n$), by \cite[Theorem 2.3.1]{Iya1}. 
Now the assumption that $A$ has additionally global dimension at most $n$ gives us that $\underline{\Hom}_A(\tau_n^{-1}(Y),X)=\Hom_A(\tau_n^{-1}(Y),X)$ by \cite[Lemma 2.4(d)]{Iya2}.
\end{proof}

\begin{lemma}
Let $N$ be an $A$-module of finite projective dimension.
Then $\pdim N= \sup \{ i \geq 0 \mid \Ext_A^i(N,A) \neq 0 \}$.
\end{lemma}
\begin{proof}
See for example \cite[Lemma VI.5.5]{ARS}.
\end{proof}

Given this lemma, it follows from the definition of $n$-representation finite algebra
that the
summands of $M$ are either projective or of projective
dimension $n$.

\begin{lemma} \label{big} Let $A$ be an $n$-representation finite algebra with $n$-cluster-tilting module $M$ and $B:=\End_A(M)$.
  \begin{enumerate}
    \item \label{point-one} Let $N$ be an indecomposable projective summand of $M$. We have $T_N \cong L_{\nu_A(N)}$ and thus $T_N$ is projective as a $B$-module.
  \item
Let $N$ be an indecomposable non-projective
summand of $M$. Suppose that its minimal projective resolution is
$$0\rightarrow P_n \rightarrow \cdots \rightarrow P_1 \rightarrow P_0 \rightarrow N \rightarrow 0.$$
Then the projective resolution as a $B$-module of $T_N$ is given by:
\begin{align}\label{Bres} 0 \rightarrow L_{\tau_n(N)} \rightarrow L_{\nu_A(P_n)} \rightarrow \cdots \rightarrow L_{\nu_A(P_0)} \rightarrow T_N \rightarrow 0.\end{align}
This then also gives an injective coresolution 
$$0 \rightarrow L_{\tau_n(N)} \rightarrow T_{P_n} \rightarrow \cdots \rightarrow T_{P_0} \rightarrow T_N \rightarrow 0$$ of $L_{\tau_n(N)}$ with $T_{P_i} \cong L_{\nu_A(P_i)}$.
\item Let $N'$ be an injective indecomposable direct summand of $M$. Then
  $L_{N'}$ is itself injective, so it is its own injective resolution.
\item \label{bij} There is a bijection between the indecomposable projective $B$-modules of injective dimension $n+1$ and the indecomposable injective $B$-modules of projective dimension $n+1$ given by $\Omega^{-(n+1)}$ with inverse $\Omega^{n+1}$.
  \end{enumerate}
\end{lemma}

  \begin{proof}\begin{enumerate}
      \item Since $N$ is projective, $T_N=D \Hom_A(N,M) \cong \Hom_A(M, \nu_A(N)) = L_{\nu_A(N)}$ by \cite[Chapter III, Corollary 6.2]{SkoYam}, so $T_N$ is itself projective.
      
      \item Since $N$ is non-projective, we have already established that
        $\pdim N=n$. 
In order to calculate $\Ext^i(N,M)$, we would apply $\Hom(-,M)$ to the projective resolution of $N$, and since
$\Ext^i(N,M)=0$ except for $i=0$ and $i=n$, we obtain an exact sequence:
\begin{align} \label{exactsequence}
0 \rightarrow \Hom_A(N,M) \rightarrow \Hom_A(P_0,M) \rightarrow \Hom_A(P_1,M) \rightarrow \cdots \rightarrow \Hom_A(P_n,M) \rightarrow \Ext_A^n(N,M) \rightarrow 0.
\end{align}
Now by the higher Auslander--Reiten formulas in Lemma \ref{higherausreiformulas}, we have
$\Ext_A^n(N,M) \cong D \Hom_A(M, \tau_n(N))$. Note also that
$D\Hom_A(N,M) =: T_N$. Thus, applying the duality $D$ to (\ref{exactsequence}) we obtain the exactness of (\ref{Bres}). It is a projective resolution
since the terms $D\Hom_A(P_i,M) \cong L_{\nu_A(P_i)}$ are projective-injective by (1), which also shows that the second exact sequence is an injective coresolution of $L_{\tau_n(N)}$.

\item This is dual to (\ref{point-one}).

\item The indecomposable projective $B$-modules which are not injective are the
  modules $L_N$ for $N$ an indecomposable direct summand of $M$ which is not
  injective. By (\ref{Bres}), we see that $\Omega^{-(n+1)} L_N=T_{\tau_n(N)}$.
  The modules of the form $T_{\tau_n(N)}$ are exactly the indecomposable
  injective $B$-modules which are not projective, and (\ref{Bres}) shows that
  $\Omega^{n+1} T_{\tau_n(N)}=L_N$.
  \end{enumerate}
  \end{proof}

In the next theorem we calculate the Coxeter transformation of higher Auslander algebras coming from $n$-cluster tilting modules in $n$-representation-finite algebras.
\begin{theorem} \label{coxetercalc}
Let $A$ be a n-representation-finite algebra with $n$-cluster-tilting module $M$, let $B=\End_A(M)$ and let $N$ be an indecomposable $A$-module in $\add(M)$.
\begin{enumerate}
\item If $N$ is projective, $C_B([L_N])= - [L_{\nu_A(N)}]$.
\item If $N$ is non-projective, let $0 \rightarrow P_n\rightarrow \cdots \rightarrow P_1 \rightarrow P_0 \rightarrow N \rightarrow 0$ be a minimal projective resolution of $N$.
Then $C_B([L_N])= \sum\limits_{i=0}^{n}{(-1)^{i+1} [L_{\nu_A(P_i)}]}+(-1)^n[L_{\tau_n(N)}]$.

\end{enumerate}

\end{theorem}
\begin{proof}
  By definition $C_B([L_N])= -[T_N]$. To express $[T_N]$ with respect to the basis of
  indecomposable projective $B$-modules, it suffices to calculate a projective
  resolution of $T_N$ and take the alternating sum. We have
  found these projective resolutions in Lemma \ref{big}, and the claim follows.
  \end{proof}

Recall that we defined the grade bijection in the Grothendieck group of an Auslander regular algebra $B$ on the indecomposable projective $A$-modules by $R_B([P_S])=[P_{\Gr_B(S)}]$, where $P_S$ denotes the projective cover of a simple module $S$ and $\Gr_B$ the grade bijection on simple modules.
We now calculate the grade bijection for higher Auslander algebras that are endomorphism rings of an $n$-cluster tilting module.

\begin{theorem}
Let $A$ be an algebra with $n$-cluster tilting module $M$ with $B:=\End_A(M)$ and let $N$ be an indecomposable module in $\add(M)$.
Then the grade bijection $R_B$ for $B$ is given by $R_B([L_N])=[L_{\nu_A(N)}]$ when $N$ is projective and $R_B([L_N])=[L_{\tau_n(N)}]$ else.

\end{theorem} 
\begin{proof}
Assume first that $N$ is projective. $S_N$ injects into $T_N \cong L_{\nu_A(N)}$, which is a summand of $B$, hence $\Hom_B(S_N,B) \neq 0$ and thus $\gr S_N =0$. Applying $D \Hom_B(-,B)$ on $L_N \rightarrow S_N \rightarrow 0$ one gets $T_N \rightarrow D \Hom_B(S_N,B) \rightarrow 0$, and since $T_N$ is projective and the middle term does not vanish, we have $\top(D \Hom_B(S_N,B)) = \top(T_N) = \top(L_{\nu_A(N)})=S_{\nu_A(N)}.$ Thus $R_B([L_N])= [L_{\nu_A(N)}]$. \newline

Assume now that $N$ is not projective. We will show that 
$D \Ext_B^i(S_N,B)= S_{\tau_n(N)}$, if $i=n+1$ and $D \Ext_B^i(S_N,B)= 0$, else.
This will show that $\gr S_N=n+1$ and $R_B([L_N])=[L_{\tau_n(N)}]$.
By the projective resolution of $S_N$ given in  (\ref{SPres}), $D \Ext_B^i(S_N,B)$ is the $i$-th homology of the complex obtained from 
$$0 \rightarrow L_{\tau_n(N)} \rightarrow \cdots L_N \rightarrow 0$$
by applying $D \Hom_B(-,B)$. This complex equals 
$$0 \rightarrow T_{\tau_n(N)} \rightarrow \cdots \rightarrow T_N \rightarrow 0,$$
which by (\ref{SIres}) for $\tau_n(N)$ is exact except at the $(n+1)$-th degree, where the homology equals $S_{\tau_n(N)}$.
\end{proof}

The last theorem tells us how to calculate the grade bijection. In the next theorem we will also use the inverse of the grade bijection, which is then given by $R_B^{-1}([L_N])=[L_{\nu_A^{-1}(N)}]$ if $N$ is injective and $R_B^{-1}([L_N])=[L_{\tau_n^{-1}(N)}]$ otherwise.
\begin{theorem} \label{maintheoremstrongnrepfin}
Let $A$ be an $n$-representation-finite algebra with $n$-cluster-tilting module $M$ and $B:=\End_A(M)$.
Then we have $(C_B R_B^{-1})^2=id $ if $n$ is even and $(C_B R_B^{-1}+id)^2=0$ when $n$ is odd.

\end{theorem} 
\begin{proof}
If $N$ is an injective indecomposable summand of $M$ then we have $C_B R_B^{-1}([L_N])=C_B([L_{\nu_A^{-1}(N)}])=-[L_N]$, since $\nu_A^{-1}(N)$ is projective.
If $N$ is an indecomposable and non-injective summand of $M$ we have:
$C_B R_B^{-1}([L_N])=C_B([L_{\tau_n^{-1}(N)}])= \sum\limits_{i=0}^{n}{(-1)^{i+1} [L_{\nu_A(P_i)}]}+(-1)^n[L_N]$, where $$0 \rightarrow P_n \rightarrow \cdots \rightarrow P_1 \rightarrow P_0 \rightarrow \tau_n^{-1}(N) \rightarrow 0$$ denotes a minimal projective resolution of $\tau_n^{-1}(N)$.

 It follows that with respect to the basis of the $[L_N]$, ordered such that the elements corresponding to the injective modules $N$ precede all the others, the matrix of $C_BR_B^{-1}$ has the block form
$$\left(\begin{array}{cc}
-I_s & E \\
0 & (-1)^nI_t
\end{array}\right)$$
for some matrix $E$, where $I_r$ denotes the $r \times r$ identity matrix, $s$ is the number of injective indecomposable modules $N$ in $\add(M)$, and $t$ is the number of the non-injective such modules. 
If $n$ is even then the square of such a matrix is the identity $I_{s+t}$, whereas if $n$ is odd then by adding the identity $I_{s+t}$ one gets the matrix
$$\left(\begin{array}{cc}
0 & E\\
0 & 0
\end{array}\right)$$
whose square is zero.
\end{proof}

Specialising to path algebras of Dynkin type gives the following special case:

\begin{example}
Let $A=KQ$ be a path algebra of Dynkin type and $M$ the direct sum of all indecomposable $A$-modules, which is a 1-cluster tilting module. Let $B:=\End_A(M)$, which is the Auslander algebra of $A$.
By Theorem \ref{maintheoremstrongnrepfin}, we have $(C_B R_B^{-1}+id)^2=0$.

\end{example}

\subsection*{Acknowledgements}


The authors would like to express their gratitude
to Sam Hopkins for communicating his conjecture
to us. This project began at a workshop on Dynamical Algebraic Combinatorics
hosted (online) by the Banff International
Research Station in 2020.
The authors are grateful to BIRS and to the organizers of
the workshop, who provided the opportunity for the project to begin. We thank the referee for their detailed comments and suggestions which significantly improved the exposition of the paper. In particular, we followed the suggestion of the referee to include a shorter proof for theorems 4.6 and 4.7.

Ren\'e Marczinzik is funded by the DFG with the project number 428999796.
Hugh Thomas is partially supported by an NSERC Discover Grant RGPIN-2016-04872
and the Canada Research Chairs program, grant number CRC-2021-00120.
Part of the work on this paper was undertaken during Emine Y\i{}ld\i{}r\i{}m's tenure as INI-Simons Post Doctoral Research Fellow hosted by the Isaac Newton Institute for Mathematical Sciences (INI) participating in programme `Cluster algebras and representation theory', and by the Department of Pure Mathematics and Mathematical Statistics (DPMMS) at the University of Cambridge.   This author would like to thank INI and DPMMS for support and hospitality during this fellowship, which was supported by Simons Foundation (Award ID 316017) and by EPSRC (grant number EP/R014604/1).

\end{document}